\documentclass[11pt]{amsart}
\usepackage{fullpage}
\usepackage{amssymb,amsfonts,amsmath}
\usepackage{enumitem} 
\newtheorem{theo}{Theorem}[section]
\newtheorem{lemma}[theo]{Lemma}
\newtheorem{conj}[theo]{Conjecture}
\newtheorem{defn}[theo]{Definition}
\newcommand{\eps}{\varepsilon}
\newcommand{\R}{\mathbb{R}}
\newcommand{\C}{{\mathcal C}}
\newcommand{\Hy}{{\mathcal H}}
\newcommand{\Part}{{\mathcal P}}
\newcommand{\ab}{{\bf a}}
\newcommand{\Qb}{{\bf Q}}
\newcommand{\1}{{\bf 1}}
\newcommand{\reld}{d^*}
\newcommand{\reldeg}{\overline{\deg}} 
\newcommand{\cJ}{\mathcal{J}} 
\newcommand{\NATS}{\mathbb{N}}

\def\noproof{{\unskip\nobreak\hfill\penalty50\hskip2em\hbox{}\nobreak\hfill%
       $\square$\parfillskip=0pt\finalhyphendemerits=0\par}\goodbreak}
\def\endproof{\noproof\bigskip}

\title{The minimum vertex degree for an almost-spanning tight cycle in a $3$-uniform hypergraph}

  \author[O. Cooley]{Oliver Cooley} \address{
    Graz University of Technology, Institute of Discrete Mathematics,
    Steyrer\-gasse~30, 8010 Graz, Austria
  }
  \email{cooley@math.tugraz.at}

  \author[R. Mycroft]{Richard Mycroft} \address{
    School of Mathematics, University of Birmingham, Edgbaston, Birmingham B15
    2TT, UK
  }
  \email{r.mycroft@bham.ac.uk}

  \thanks{RM was supported by EPSRC grant EP/M011771/1.
  }
\date{\today}

\begin{document}

\begin{abstract}
We prove that any $3$-uniform hypergraph whose minimum vertex degree is at least $\left(\frac{5}{9} + o(1) \right)\binom{n}{2}$ admits an almost-spanning tight cycle, that is, a tight cycle leaving $o(n)$ vertices uncovered. The bound on the vertex degree is asymptotically best possible. Our proof uses the hypergraph regularity method, and in particular a recent version of the hypergraph regularity lemma proved by Allen, B\"ottcher, Cooley and Mycroft. 
\end{abstract}
\maketitle

\section{Introduction}

The study of Hamilton cycles in graphs and hypergraphs has been an active area of research for many years, going back to Dirac's celebrated theorem~\cite{dirac} that any graph on $n \geq 3$ vertices with minimum degree at least $n/2$ admits a Hamilton cycle. Finding analogues of this theorem for $k$-graphs (\emph{i.e.} $k$-uniform hypergraphs) is one of the major directions of recent research in this area. To discuss these results we use several commonly-used standard terms, definitions of which can be found in Section~\ref{sec:defs}. For a more expository presentation of recent research in this area we refer the reader to the surveys of K\"uhn and Osthus~\cite{KOSurvey}, R\"odl and Ruci\'nski~\cite{RRSurvey} and Zhao~\cite{ZhaoSurvey}.

The first analogues of Dirac's theorem for $k$-graphs were expressed in terms of minimum codegree, beginning with the work of Katona and Kierstead~\cite{hh-katona99} who established the first non-trivial bounds on the codegree Dirac threshold for a tight Hamilton cycle in a $k$-graph for $k \geq 3$. R\"odl, Ruci\'nski and Szemer\'edi~\cite{hh-rodl06, hh-rodl08} then improved this bound by determining the asymptotic value of this threshold, first for $k=3$ and then for any $k \geq 3$. The asymptotic codegree Dirac threshold for an $\ell$-cycle for any $1 \leq \ell < k$ such that $\ell$ divides $k$ follows as a consequence of this. This left those values of $\ell$ for which $k-\ell$ does not divide $k$, in which cases the Dirac threshold was determined asymptotically through a series of works by K\"uhn and Osthus~\cite{pp-kuhn06-cherry}, Keevash, K\"uhn, Mycroft and Osthus~\cite{hh-keevash11}, H\`an and Schacht~\cite{hh-han10} and K\"uhn, Mycroft and Osthus~\cite{KMO14}. These results can all be collectively described by the following theorem, whose statement gives the asymptotic codegree Dirac threshold for any $k$ and $\ell$. 

\begin{theo}[\cite{hh-han10, hh-keevash11, KMO14, pp-kuhn06-cherry, hh-rodl06, hh-rodl08}] \label{codeg}
  For any $k \geq 3$, $1 \leq \ell < k$ and $\eta > 0$, there exists~$n_0$ such that if $n \geq
  n_0$ is divisible by $k-\ell$ and $H$ is a $k$-graph on $n$ vertices with $$\delta(H) \geq 
  \begin{cases}
    \left( \frac{1}{2} + \eta \right) n& \mbox{ if $k-\ell $ divides $k$,} \\
    \left(\frac{1}{\lceil 
    \frac{k}{k-\ell} \rceil(k-\ell)}+\eta\right) n & \mbox{otherwise,}
  \end{cases} 
  $$ 
  then $H$ contains a Hamilton $\ell$-cycle. Moreover, in each case this condition is best possible up to the $\eta n$ error term.
\end{theo}

More recently the exact codegree Dirac threshold (for large $n$) has been identified in some cases, namely for $k=3, \ell=2$ by R\"odl, Ruci\'nski and Szemer\'edi~\cite{RoRuSz09}, for $k=3, \ell=1$ by Czygrinow and Molla~\cite{CM}, for any $k \geq 3$ and $\ell < k/2$ by Han and Zhao~\cite{HZ}, and for $k=4$ and $\ell=2$ by Garbe and Mycroft~\cite{GaMy16}.

For other types of degree conditions much less is known. Bu\ss, H\`an and Schacht~\cite{BHS13} identified asymptotically the vertex degree Dirac threshold for a loose Hamilton cycle in a $3$-graph, following which Han and Zhao~\cite{hh-han14} established this threshold exactly (for large $n$). Very recently the asymptotic $(k-2)$-degree Dirac threshold for a Hamilton $\ell$-cycle in a $k$-graph for $1 \leq \ell < k/2$ was determined by Bastos, Mota, Schacht, Schnitzer and Schulenburg~\cite{BMSSS16}. 

Perhaps the most prominent outstanding open problem is to identify the vertex degree Dirac threshold for a tight Hamilton cycle in a $3$-graph. There is a qualitative difference between this problem and the results described above, in that vertex degree conditions only give information about edges containing a given vertex (in particular, many pairs may have degree zero), but a tight cycle requires us to find edges intersecting in pairs. Possibly due to this, progress on this problem has been slower. Glebov, Person and Weps~\cite{GPW} gave the first non-trivial upper bound on this threshold, showing it to be at most $(1-\eps)\binom{n}{2}$, where $\eps \approx 5 \times 10^{-7}$. R\"odl and Ruci\'nski~\cite{RR} subsequently improved this bound to around $0.92\binom{n}{2}$, and very recently R\"odl, Ruci\'nski, Schacht and Szemer\'edi~\cite{RRSSz} strengthened this result by showing that the the threshold is at most $0.8\binom{n}{2}$. However, this is still some way from the conjectured value of $(\tfrac{5}{9} + o(n)) \binom{n}{2}$ (which matches the lower bound provided by the best known construction).

\begin{conj}[\cite{RRSurvey}, Conjecture 2.18]\label{conj}
For every $\eta > 0$, there exists $n_0$ such that if $H$ is a $3$-graph on $n \geq n_0$ vertices with $\delta(H) \geq \left(\frac{5}{9} + \eta \right) \binom{n}{2}$, then $H$ contains a tight Hamilton cycle.
\end{conj}

(Actually, the original conjecture was that for any $k \geq 3$ the vertex degree threshold for a tight Hamilton cycle in a $k$-graph is asymptotically equal to that which forces a perfect matching. However, this was disproved for $k \geq 4$ by Han and Zhao~\cite{HZ2}). The main result of this paper is the following theorem, which states that the assumptions of Conjecture~\ref{conj} suffice to ensure the existence of an almost-spanning tight cycle.

\begin{theo}\label{main}
For every $\eta > 0$ there exists $n_0$ such that if $H$ is a $3$-graph on $n \geq n_0$ vertices with $\delta(H) \geq \left(\frac{5}{9} + \eta \right) \binom{n}{2}$ then $H$ contains a tight cycle of length at least $(1-\eta) n$. 
\end{theo}

The following example shows that Conjecture~\ref{conj} and Theorem~\ref{main} are asymptotically best possible. Fix $\eta > 0$, take disjoint sets $A$ and $B$ with $|A| =  \lfloor\tfrac{(1-\eta) n-1}{3}\rfloor$ and $|B| = n-|A|$, and take $H$ to be the $3$-graph on vertex set $A \cup B$ whose edges are all triples which intersect $A$. Then for any tight cycle $C$ in $H$, any three consecutive vertices in $C$ must include a vertex of $A$, so the length of $C$ is at most $3|A| < (1-\eta)n$. However, it is easily checked that $\delta(H) = \binom{n-1}{2} - \binom{|B|-1}{2} \approx (\tfrac{5}{9}-\tfrac{4}{3}\eta - \eta^2)\binom{n}{2}$.

Many of the Dirac thresholds mentioned above were established by absorbing arguments consisting of two parts: a long cycle lemma which states that the given $k$-graph contains an almost-spanning cycle, and an absorbing lemma which allows us to include in this cycle a so-called absorbing path, which can `absorb' the small set of leftover vertices to transform the almost-spanning cycle into a Hamilton cycle. In this way Theorem~\ref{main} represents a significant step towards a full proof of Conjecture~\ref{conj}. Furthermore, the proof of Theorem~\ref{main} provides a clear illustration of how the the recent `Regular Slice Lemma' of Allen, B\"ottcher, Cooley and Mycroft~\cite{ABCM} may be used to prove embedding results for uniform hypergraphs. In our opinion the proof of Theorem~\ref{main} is significantly more concise and notationally simpler than it would have been using previous hypergraph regularity methods.

This paper is organised as follows. In Section~\ref{sec:defs} we give the necessary definitions, then in Section~\ref{sec:matchings} we prove that any $3$-graph satisfying the minimum vertex degree condition of Theorem~\ref{main} admits a tightly-connected perfect fractional matching (Lemma~\ref{fracmatch}). In Section~\ref{sec:regandcycle} we introduce the hypergraph regularity theory we use from~\cite{ABCM}, including the Regular Slice Lemma, which returns a `reduced $k$-graph' which almost inherits the minimum vertex degree of the original $k$-graph, and the Cycle Embedding Lemma, which given a large tightly-connected fractional matching in the reduced $k$-graph, returns a long tight cycle in the original $k$-graph. Finally, in Section~\ref{sec:mainproof} we combine the results of the previous two sections to complete the proof of Theorem~\ref{main}.

\section{Definitions and Notation} \label{sec:defs}

A \emph{$k$-graph} $H$ consists of a set of vertices $V(H)$ and a set of edges $E(H)$, where each edge consists of $k$ vertices. So a $2$-graph is a simple graph. We write $v(H) := |V(H)|$ and $e(H) := |E(H)|$ for the number of vertices and edges respectively. Also, we write $e \in H$ to mean that $e$ is an edge of $H$. Given any integer $1 \leq \ell < k$, we say that a $k$-graph $C$ is an \emph{$\ell$-cycle} if the vertices of $C$ may be cyclically ordered in such a way that every edge of $C$ consists of $k$ consecutive vertices and each edge intersects the subsequent edge (in the natural ordering of the edges) in precisely $\ell$ vertices. In particular, $(k-1)$-cycles and $1$-cycles are known as \emph{tight cycles} and \emph{loose cycles} respectively. A \emph{Hamilton $\ell$-cycle} in a $k$-graph $H$ is a spanning subgraph of $H$ which is an $\ell$-cycle; similarly as before we also speak of tight Hamilton cycles and loose Hamilton cycles. The \emph{length} of a cycle is the number of edges it contains; note that for tight cycles this is identical to the number of vertices. 

Given a $k$-graph $H$ and a set $S \subseteq V(H)$, the \emph{degree} of $S$,
denoted $d_H(S)$ (or $d(S)$ when $H$ is clear from the context), is the number
of edges of $H$ which contain $S$ as a subset. For $1\leq s < k$ the
\emph{minimum $s$-degree of $H$}, denoted $\delta_s(H)$, is then defined to be
the minimum of $d(S)$ taken over all sets $S \subseteq V(H)$ of size $s$.
In particular, we refer to the minimum $1$-degree as the \emph{minimum vertex degree}
of $H$, and to the minimum $(k-1)$-degree as the \emph{minimum codegree} of $H$.
The principal notion of minimum degree used in this paper is that of vertex degree,
so we adopt the convention that $\delta(H)$ (without the subscript) will always
refer to the minimum vertex degree of $H$ (note that this coincides with the usual
notion of minimum degree for graphs). For any vertex $v \in V(H)$, the \emph{link graph}
of $v$, denoted $L_H(v)$ (or $L(v)$ if $H$ is clear from the context) is the
$(k-1)$-graph on vertex set $V(H)$ whose edges are all $(k-1)$-tuples $S$ for
which $S \cup \{v\} \in H$.  (Note: it will be convenient later on that we did
\emph{not} delete the vertex $v$ from its link graph even though by definition
it is not contained in any edges.)

A \emph{matching} in a $k$-graph $H$ is a set of disjoint edges of $H$;
it is \emph{perfect} if it covers every vertex of $H$ (in other words, if it has
precisely $v(H)/k$ edges). The LP relaxation of a matching is a \emph{fractional matching};
this is an assignment of a weight $w_e \in [0, 1]$ to each edge $e \in H$ such that
for any vertex $v \in V(H)$ we have $\sum_{e \ni v} w_e \leq 1$.
The \emph{total weight} of a fractional matching is $\sum_{e \in H} w_e$; similarly
as for integer matchings we say that a fractional matching is \emph{perfect} if it
has total weight $v(H)/k$.

Given edges $e$ and $e'$ of a $k$-graph $H$, a \emph{tight walk}
from $e$ to $e'$ is a sequence of not-necessarily-distinct edges
$e = e_0, e_1 \dots, e_{\ell-1}, e_\ell = e'$ such that $|e_i \cap e_{i-1}| = k-1$
for each $i \in [\ell]$. It is not hard to see that this gives an equivalence
relation on the edge set of $H$. The \emph{tight components} of $H$ are the
equivalence classes of this relation. We say that $H$ is \emph{tightly connected}
if it has only one tight component, that is, if there is a tight walk between any
pair of edges. Moreover, we say that a fractional matching in $H$ is
\emph{tightly-connected} if all edges of non-zero weight lie in the same tight component of $H$. 

We write $x \ll y$ to mean that for any $y > 0$ there exists $x_0 > 0$ such that the subsequent calculations hold for any $0 < x \leq x_0$. Similar statements with more variables are defined similarly. We write $x = y \pm z$ to mean that $y-|z| \leq x \leq y+|z|$, and $[n]$ to denote the set of integers from $1$ to $n$. We omit floors and ceilings wherever they do not affect the argument.

\section{Tightly-Connected Perfect Fractional Matchings}\label{sec:matchings}

In this section we show that any $3$-graph meeting the appropriate
minimum vertex degree condition admits a tightly-connected perfect fractional matching.
For this we use the following theorem of Erd\H{o}s and Gallai~\cite{EG}, which gives a tight
lower bound on the smallest possible size of a maximum matching in a graph of given
order and size.

\begin{theo}[\cite{EG}]\label{erdosgallai}
Any graph $G$ on $N$ vertices with $e(G) > \max \left\{\binom{2k-1}{2}, \binom{k-1}{2} + (k-1)(N-k+1) \right\}$ admits a matching of size $k$.
\end{theo}

Our next lemma shows that the largest connected component in a graph of density
higher than $5/9$ covers at least two-thirds of the vertices and contains a large
matching, and moreover that
the largest connected components of two such graphs on the same vertex set must
share an edge. This lemma will later be applied to link graphs of vertices.

\begin{lemma} \label{graphmeet}
Let $G_1$ and $G_2$ be graphs on a common vertex set $V$ of size $n$, where $3$ divides $n$, such that $e(G_1), e(G_2) > \frac{5}{9}\binom{n}{2}$. Let $C_1$ and $C_2$ be largest connected components in $G_1$ and $G_2$ respectively. Then the following statements hold.
\begin{enumerate}[label=(\roman*)]
\item $v(C_i) > 2n/3$ for $i=1,2$,
\item $e(C_i) > \frac{4}{9}\binom{n}{2}$ for $i=1,2$,
\item $C_i$ contains a matching of size $n/3$ for $i=1,2$, and
\item $C_1$ and $C_2$ have an edge in common. 
\end{enumerate}
\end{lemma}

\begin{proof} In this proof we make repeated use of the fact that for $0 < x < 1$ we have $\binom{xn}{2} < x^2\binom{n}{2}$.
For (i), suppose for a contradiction that every connected component in $G_i$
has at most $2n/3$ vertices. Then we may form disjoint sets $A$ and $B$ such
that $V = A \cup B$, such that $A$ and $B$ can each be written as a union of
connected components of $G_i$, and such that $|A|, |B| \leq 2n/3$. We then have
$$e(G_i) \leq \binom{|A|}{2} + \binom{|B|}{2} \leq \binom{2n/3}{2} + \binom{n/3}{2} < \frac{5}{9} \binom{n}{2},$$
giving the desired contradiction to prove (i). From this (ii) follows immediately,
since the number of edges of $G_i$ which are not in $C_i$ is at most
$\binom{n-v(C_i)}{2} < \binom{n/3}{2} < \frac{1}{9}\binom{n}{2}$.

For (iii), let $x$ be such that $v(C_i) = (1-x)n$, so $x$ is the proportion of vertices of $V$
which are not in $C_i$. In particular $0 \leq x < \tfrac{1}{3}$ by (i).
Observe that at most $\binom{xn}{2} \leq x^2\binom{n}{2}$ edges of $G_i$ are not in $C_i$, so $C_i$ has more than $\left(\frac{5}{9} - x^2 \right) \binom{n}{2}$ edges.
It is easily checked that the inequality 
$$\left(\frac{5}{9} - x^2 \right) \binom{n}{2} > \max \left\{\binom{2 \cdot \frac{n}{3}-1}{2}, \binom{\tfrac{n}{3}-1}{2} + \left(\frac{n}{3}-1\right) \left((1-x)n-\frac{n}{3}+1\right) \right\}$$ 
holds for any $0 \leq x <\tfrac{1}{3}$,
so by Theorem~\ref{erdosgallai}
(with $(1-x)n$ and $n/3$ in place of $N$ and $k$ respectively) $C_i$ admits a matching
$M$ of size $n/3$ as claimed.

For (iv), fix $\alpha$ and $\beta$ so that $|V(C_1)| = (1-\alpha)n$ and
$|V(C_2)| = (1 - \beta) n$. So $0 \leq \alpha, \beta < 1/3$ and
$|V(C_1) \cap V(C_2)| \geq (1 - \alpha - \beta) n$. Similarly as before, at most
$\binom{\alpha n}{2} \leq \alpha^2\binom{n}{2}$ edges of $G_1$ are not in $C_1$, and at most
$\binom{\beta n}{2} \leq \beta^2\binom{n}{2}$ edges of $G_2$ are not in $C_2$. Now suppose for a
contradiction that $C_1$ and $C_2$ have no edges in common. Then we have
\begin{align*}
\left(\frac{5}{9}-\alpha^2\right)\binom{n}{2} + \left(\frac{5}{9} - \beta^2\right) \binom{n}{2}
& < e(C_1) + e(C_2) \\
& \leq \binom{v(C_1)}{2} + \binom{v(C_2)}{2} - \binom{|V(C_1) \cap V(C_2)|}{2} \\ 
& \leq \binom{(1-\alpha) n}{2} + \binom{(1-\beta) n}{2} - \binom{(1- \alpha - \beta)n}{2} \\
& = \left((1-\alpha)^2 + (1-\beta)^2 - (1-\alpha-\beta)^2\right)\binom{n}{2}\\
&\hspace{1cm} - \tfrac{n}{2}\left(\alpha(1-\alpha)+\beta(1-\beta)-(\alpha+\beta)(1-\alpha - \beta)\right)\\
& \leq (1-2\alpha\beta) \binom{n}{2}.
\end{align*}
So $1/9 < \alpha^2 + \beta^2 - 2 \alpha \beta = (\alpha - \beta)^2$, which implies that $|\alpha - \beta| > 1/3$, contradicting the fact that $0 \leq \alpha, \beta < 1/3$. We conclude that $C_1$ and $C_2$ must have an edge in common, as required.
\end{proof}

Turning to $3$-graphs, we prove the following lemma, the second part of which was the principal aim of this section. Note that the requirement that $3$ divides $n$ in Lemmas~\ref{graphmeet} and~\ref{fracmatch} is for simplicity only, and can easily be removed by (for example) using a fractional version of Theorem~\ref{erdosgallai}.

\begin{lemma}\label{fracmatch}
Let $\eps > 0$, and let $H$ be a $3$-graph on $n$ vertices with $\delta(H) > \frac{5}{9}\binom{n}{2}$. Suppose also that $3$ divides $n$. Then 
\begin{enumerate}[label=(\roman*)]
\item $H$ has a tightly-connected spanning subgraph $H'$ with $\delta(H') \geq \frac{4}{9}\binom{n}{2}$, and
\item $H'$ admits a perfect fractional matching (which is therefore a tightly-connected perfect fractional matching in $H$).
\end{enumerate} 
\end{lemma}

\begin{proof}
Write $V := V(H)$. For (i), observe that for any vertex $u \in V$, the link graph
$L(u)$ has $n$ vertices and more than $\frac{5}{9}\binom{n}{2}$ edges.
For each $u \in V$ let $C_u$ be the largest connected component of $L(u)$, and let
$C_u^*$ denote $\{e \in H: e-v \in C_u\}$, i.e.\ the edges of $C_u$ with $u$ added back in.
First observe that $C_u^*$ lies within a single tight component of $H$.
Next observe that
by Lemma~\ref{graphmeet}, for each vertex $u \in V$ we have $|v(C_u)| > 2n/3$
and $e(C_u) > \frac{4}{9}\binom{n}{2}$, and moreover
$C_u$ and $C_v$ have an edge in common for any $u, v \in V$. The latter implies
that $C_u^*$ and $C_v^*$ lie in the same tight component of $H$, and more generally
all the sets $C_u^*$ for $u \in V$ are contained within a single tight component
$H'$ of $H$. So $H'$ is a tightly-connected spanning subgraph $H'$ of $H$; since
$C_u^* \subseteq H'$ for any $u \in V$ we also have
$\delta(H') \geq \frac{4}{9}\binom{n}{2}$.

For (ii), suppose for a contradiction that there is no such matching. Then (after fixing some order of $V$) by a
standard application of Farkas' lemma there exists a vector $\ab \in \R^n$ such
that $\ab \cdot \1 > 0$ and $\ab \cdot \chi(e) \leq 0$ for every $e \in H'$
(where for a set $S \subset V$, $\chi(S)$ denotes the characteristic vector of $S$, whose $i$th coordinate is one if the $i$th vertex of $V$ is in $S$, and zero otherwise).
Fix such an $\ab$, and choose $u \in V$ for which $\ab \cdot \chi(\{u\})$ is maximal.
By Lemma~\ref{graphmeet}~(iii), $C_u$ contains a matching $M$ of size $n/3$.
Let $x_iy_i$ for $i \in [n/3]$ be the edges of $M$, and let
$z_1, \dots, z_{n/3}$ be the vertices of $V$ not covered by $M$. Define
$e_i := \{u, x_i, y_i\}$ and $S_i := \{x_i, y_i, z_i\}$ for each $i \in [n/3]$.
Then the sets $S_i$ partition $V$, so $\sum_{i \in [n/3]} \chi(S_i) = \chi(V) = \1$.
Moreover, since $M$ was a matching in $C_u$ each $e_i$ is an edge of $H'$, and by
choice of $u$ we have $\ab \cdot \chi(e_i) \geq \ab \cdot \chi(S_i)$ for each
$i \in [n/3]$. It follows that 
$$ 0 < \ab \cdot \1 = \sum_{i \in [n/3]} \ab \cdot \chi(S_i) \leq \sum_{i \in [n/3]} \ab \cdot \chi(e_i) \leq 0,$$
where the first and last inequalities hold by choice of $\ab$. This gives the desired contradiction.
\end{proof}

\section{Tight Cycles in Regular Slices} \label{sec:regandcycle}

As described in the introduction, our proof uses the recent `Regular Slice Lemma' of Allen, B\"ottcher, Cooley and Mycroft~\cite{ABCM}. This was derived from the Strong Hypergraph Regularity Lemma of R\"odl and Schacht~\cite{RS1}, and allows for the application of hypergraph regularity with less notational complexity. We introduce this lemma in Subsection~\ref{subsec:rhgraph}, but first give the necessary definitions and notation in Subsection~\ref{subsec:reg}. Most of these definitions are in the form in which they appear in~\cite{ABCM}, which in turn was based on the framework of R\"odl and Schacht~\cite{RS1, RS2}. 

\subsection{Hypergraph regularity} \label{subsec:reg}
A \emph{hypergraph} $H$ consists of a vertex set $V$ and a set of edges $E$, where each edge is a subset of $V$. In particular, $H$ is a \emph{complex} if $E$ is down-closed,
meaning that whenever $e\in E$ and $e'\subseteq e$ we have $e'\in E$. All complexes considered here have the property that $\{v\} \in E$ for every $v \in V$.
A \emph{$k$-complex} is a complex in which all edges have cardinality at most $k$.
Given a complex $\Hy$, we use $\Hy^{(i)}$ to denote the $i$-graph obtained by
taking all vertices of $\Hy$ and those edges of cardinality exactly $i$.

Let $\Part$ partition a vertex set $V$ into parts $V_1, \dots, V_s$. Then we say
that a subset $S \subseteq V$ is \emph{$\Part$-partite} if $|S \cap V_i| \leq 1$ for
every $i \in [s]$. Similarly, we say that a hypergraph $\Hy$ is \emph{$\Part$-partite} if
all of its edges are $\Part$-partite. In this case we refer to the parts of
$\Part$ as the \emph{vertex classes} of $\Hy$. A hypergraph $\Hy$ is \emph{$s$-partite} if there is some partition $\Part$ of $V(\Hy)$ into $s$ parts for which $\Hy$ is $\Part$-partite. 

Let $\Hy$ be a $\Part$-partite hypergraph.
Then for any $A \subseteq [s]$ we write $V_A$ for $\bigcup_{i \in A} V_i$. The
\emph{index} of a $\Part$-partite set $S \subseteq V$ is $i(S) := \{i \in [s] : |S \cap
V_i| = 1\}$. We write $\Hy_A$ to denote the collection of edges in $\Hy$ with
index $A$. So $\Hy_A$ can be regarded as an $|A|$-partite $|A|$-graph on vertex
set $V_A$, with vertex classes $V_i$ for $i \in A$. It is often convenient to refer to the subgraph induced by a set of vertex classes rather than with a given index; if $X$ is a $k$-set of vertex classes of $\Hy$ we write $\Hy_X$ for the $k$-partite $k$-uniform subgraph of $\Hy^{(k)}$ induced by $\bigcup X$, whose vertex classes are the members of~$X$. Note that $\Hy_X = \Hy_{\{i : V_i \in X\}}$. In a similar manner we write $\Hy_{X^<}$ for the $k$-partite hypergraph on vertex set $\bigcup X$ whose edge set is $\bigcup_{X' \subsetneq X} \Hy_X$. Note that if $\Hy$ is a complex, then $\Hy_{X^<}$ is a $(k-1)$-complex because $X$ is a $k$-set.

Let $i \geq 2$, let $\Part_i$ be a partition of a vertex set $V$ into $i$ parts,
let $H_i$ be any $\Part_i$-partite $i$-graph, and let $H_{i-1}$ be any $\Part_i$-partite
$(i-1)$-graph, on the common vertex set $V$. We denote by
$K_i(H_{i-1})$ the $\Part_i$-partite $i$-graph on $V$ whose edges are all $\Part_i$-partite $i$-sets in $V$ which are supported on $H_{i-1}$ (i.e.~induce a copy of the complete $(i-1)$-graph~$K_i^{i-1}$
on~$i$ vertices in~$H_{i-1}$).
The \emph{density of $H_i$ with respect to
$H_{i-1}$} is then defined to be \[ d(H_i|H_{i-1}):= \frac{|K_i(H_{i-1})\cap
H_i|}{|K_i(H_{i-1})|}\]
 if
$|K_i(H_{i-1})|>0$. For convenience we take
$d(H_i|H_{i-1}):=0$ if $|K_i(H_{i-1})| = 0$. So $d(H_i|H_{i-1})$ is the proportion of
$\Part_i$-partite copies of
$K^{i-1}_i$ in $H_{i-1}$ which are also edges of $H_i$. When $H_{i-1}$ is clear from the
context,
we simply refer to $d(H_i | H_{i-1})$ as the \emph{relative density of $H_i$}.
More generally, if $\Qb:=(Q_1,Q_2,\ldots,Q_r)$ is a collection
of~$r$ not necessarily disjoint subgraphs of~$H_{i-1}$, we define $K_i(\Qb):=\bigcup_{j=1}^r K_i(Q_j)$
and \[d(H_i |\Qb):= \frac{|K_i(\Qb)\cap H_i|}{|K_i(\Qb)|}\] if
$|K_i(\Qb)|>0$. Similarly as before we take $d(H_i |\Qb):=0$ if
$|K_i(\Qb)| =0$. 
We say that $H_i$ is \emph{$(d_i,\eps,r)$-regular with respect
to~$H_{i-1}$} if we have $d(H_i|\Qb) = d_i \pm \eps $ for every $r$-set $\Qb$
of subgraphs of $H_{i-1}$ such that $|K_i(\Qb)| > \eps |K_i(H_{i-1})|$. We
often refer to $(d_i,\eps,1)$-regularity simply as
\emph{$(d_i,\eps)$-regularity}; also, we say simply that $H_i$ is \emph{$(\eps,r)$-regular with respect
to~$H_{i-1}$} to mean that there exists some $d_i$ for which $H_i$ is \emph{$(d_i,\eps,r)$-regular with respect
to~$H_{i-1}$}. Finally, given an $i$-graph $G$
whose vertex set contains that of $H_{i-1}$, we say that $G$ is
\emph{$(d_i,\eps,r)$-regular with respect to~$H_{i-1}$} if the $i$-partite subgraph of
$G$ induced by the vertex classes of $H_{i-1}$ is $(d_i,\eps,r)$-regular with respect to
$H_{i-1}$. Similarly as before, when $H_{i-1}$ is clear from the context,
we refer to the density of this $i$-partite subgraph of $G$ with respect
to $H_{i-1}$ as the \emph{relative density of~$G$}.

Now let $\Hy$ be an $s$-partite $k$-complex on
vertex classes $V_1, \dots, V_s$, where $s \geq k \geq 3$.
So if $e \in \Hy^{(i)}$ for some $2 \leq i \leq k$, then the vertices of $e$ induce a copy of
$K^{i-1}_i$ in $\Hy^{(i-1)}$. This means that for any index $A \in \binom{[s]}{i}$ the
density $d(\Hy^{(i)}[V_A]|\Hy^{(i-1)}[V_A])$ can be regarded as the proportion of
`possible edges' of $\Hy^{(i)}[V_A]$ which are indeed edges. (Here a `possible edge' is a
subset of $V(\Hy)$ of index $A$ all of whose proper subsets are edges of~$\Hy$).
We therefore say that~$\Hy$ is \emph{$(d_k,\dots,d_2,\eps_k,\eps,r)$-regular} if
\begin{enumerate}[label=(\alph*)]
 \item for any $2 \leq i \leq k-1$ and any $A \in
  \binom{[s]}{i}$, the induced subgraph $\Hy^{(i)}[V_A]$ is
  $(d_i,\eps)$-regular with respect to~$\Hy^{(i-1)}[V_A]$, and
 \item for any $A \in \binom{[s]}{k}$, the induced subgraph 
  $\Hy^{(k)}[V_A]$ is $(d_k, \eps_k, r)$-regular with respect to~$\Hy^{(k-1)}[V_A]$.
\end{enumerate}
So each constant $d_i$ approximates the relative density of each subgraph $\Hy^{(i)}[V_A]$ for $A \in \binom{[s]}{i}$ for which $\Hy^{(i)}[V_A]$ is non-empty.
For a $(k-1)$-tuple $\mathbf{d} = (d_k, \dots, d_2)$ we write
$(\mathbf{d},\eps_k,\eps,r)$-regular to mean
$(d_k,\dots,d_2,\eps_k,\eps,r)$-regular. 
A regular complex is the
correct notion of `approximately random' for hypergraph regularity.

\subsection{Regular slices and reduced $k$-graphs}\label{subsec:rhgraph}
The Regular Slice Lemma says that any $k$-graph $G$ admits a
`regular slice'. This is a multipartite $(k-1)$-complex $\cJ$ whose vertex
classes have equal size, which is regular, and which moreover has the property that $G$ is regular with
respect to $\cJ$. The first two of these conditions are formalised in the
following definition: we say that a $(k-1)$-complex $\cJ$ is
\emph{$(t_0,t_1,\eps)$-equitable} if it has the following properties.
\begin{enumerate}[label=(\alph*)]
  \item $\cJ$ is $\Part$-partite for some $\Part$ which partitions $V(\cJ)$ into
  $t$ parts, where $t_0\le t\le t_1$, of equal size. We
  refer to $\Part$ as the \emph{ground partition} of $\cJ$, and to the parts of $\Part$ as the
  \emph{clusters} of $\cJ$.
  \item There exists a \emph{density vector} $\mathbf{d}=(d_{k-1},\ldots,d_2)$
  such that for each $2\le i\le k-1$ we have $d_i\ge 1/t_1$ and
  $1/d_i\in\NATS$, and the $(k-1)$-complex $\cJ$ is $(\mathbf{d},\eps,\eps,1)$-regular.
\end{enumerate}
For any $k$-set $X$ of clusters of $\cJ$, we write $\hat{\cJ}_X$ for the
$k$-partite $(k-1)$-graph $\cJ^{(k-1)}_{X^<}$; we refer to $\hat{\cJ_X}$ as a
\emph{polyad}.
Given a $(t_0,t_1,\eps)$-equitable $(k-1)$-complex $\cJ$ and a
$k$-graph $G$ on $V(\cJ)$, we say that \emph{$G$ is $(\eps_k,r)$-regular with respect to a $k$-set $X$ of clusters of $\cJ$} if there exists some $d$ such
that $G$ is $(d,\eps_k,r)$-regular with respect to the polyad $\hat{\cJ_X}$. We
also write $\reld_\cJ(X)$ for the relative density of $G$ with respect to
$\hat{\cJ_X}$, or simply $\reld(X)$ if $\cJ$ is clear from the context, which will usually be the case in applications. 

We can now present the definition of a regular slice.

\begin{defn}[Regular slice] Given $\eps,\eps_k>0$,
$r,t_0,t_1\in\NATS$, a $k$-graph $G$ and a $(k-1)$-complex $\cJ$ on $V(G)$, we
call $\cJ$ a $(t_0,t_1,\eps,\eps_k,r)$-regular slice for $G$ if
$\cJ$ is $(t_0,t_1,\eps)$-equitable and $G$ is $(\eps_k,r)$-regular with respect to all but at most
$\eps_k\binom{t}{k}$ of the $k$-sets of clusters of $\cJ$, where $t$ is the number of clusters of $\cJ$.
\end{defn}

If we specify the density vector $\mathbf{d}$ and the number of clusters $t$ of an equitable complex or a regular slice, then it is not necessary to specify $t_0$ and $t_1$ (since the only role of these is to bound $\mathbf{d}$ and $t$). In this situation we write that $\cJ$ is $(\cdot, \cdot, \eps)$-equitable, or is a $(\cdot,\cdot,\eps,\eps_k,r)$-regular slice for $G$. 

Given a regular slice $\cJ$ for a $k$-graph $G$, we use a weighted reduced graph to record the relative densities $\reld(X)$ for $k$-sets $X$ of clusters of $\cJ$; this is defined as follows.

\begin{defn}[Weighted reduced $k$-graph]
 Given a $k$-graph $G$ and a $(k-1)$-complex $\cJ$ on $V(G)$ which is a $(t_0,t_1,\eps,\eps_k,r)$-regular slice for $G$, we define the \emph{weighted reduced $k$-graph of $G$}, denoted $R(G)$, to be the complete weighted $k$-graph whose
 vertices are the clusters of $\cJ$, and where each edge $X$ is given weight $\reld(X)$
 (so in particular, the weight is in $[0,1]$). Note that $R(G)$ does depend on $\cJ$, but this will always be clear from the context.
\end{defn}

Essentially, the Regular Slice Lemma states that for any $k$-graph $G$ we may choose a regular slice $\cJ$ for $G$ such that the weighted reduced $k$-graph of $G$ with respect to $\cJ$ inherits various properties from $G$. The inherited properties of the full version of the lemma include densities of small subgraphs, degree conditions, and vertex neighbourhoods. However, for our purposes here we only require $R_J(G)$ to inherit vertex degree conditions from $G$, so we omit the other properties and refer the reader to~\cite{ABCM} for the full statement. To describe inheritance of degree conditions, it is easiest to use the following notion of relative degree.

Let $G$ be a $k$-graph on $n$ vertices. For a vertex $v \in V(G)$, the \emph{relative degree of $v$ in $G$} is defined to be $\reldeg(v;G) := d(v)/\binom{n-1}{k-1}$. Similarly, if $G$ is instead a weighted $k$-graph with weight function $\reld$, then we define 
\[\reldeg(v;G) := \frac{\sum_{e \in G : S \subseteq e} \reld(e)}{\binom{n-1}{k-1}}\,.\]
In other words, $\reldeg(v;G)$ is the (weighted) proportion of $k$-sets of vertices of $G$ extending $S$ which are in fact edges of $G$. Finally, for any set $S \subseteq V(G)$ we define the \emph{mean relative degree of $S$ in $G$}, denote $\reldeg(S; G)$, to be the mean average of $\reldeg(v;G)$ over
all $v \in S$.

We can now give the form of the Regular Slice Lemma which we use in the proof of Theorem~\ref{main}. 
%As well as providing many extra properties of $R(G)$, the full statement in~\cite{ABCM} also allows us to simultaneously apply the lemma to several $k$-graphs, and to provide an initial partition to be refined, but we do not need either of these strengthenings here. Loosely speaking, the proof of the Regular Slice Lemma proceeded by first applying the Strong Hypergraph Regularity Lemma of R\"odl and Schacht~\cite{RS1}, and then randomly choosing cells from the so-called `$(k-1)$-family of partitions' obtained from this to form a single $(k-1)$-complex.

\begin{lemma}\cite[Lemma~6 (Regular Slice Lemma)]{ABCM}\label{lem:simpreg}
Let $k \geq 3$ be a fixed integer. For all positive integers $t_0$,
positive~$\eps_k$ and all functions $r: \NATS \rightarrow \NATS$ 
and $\eps: \NATS \rightarrow (0,1]$, there are integers~$t_1$ 
and~$n_2$ such that the following holds for all $n \ge n_2$ which are divisible 
by~$t_1!$. 
Let $G$ be a $k$-graph whose vertex set $V$ has size $n$. 
Then there exists a $(k-1)$-complex
$\cJ$ on $V$ which is a
$(t_0,t_1,\eps(t_1),\eps_k,r(t_1))$-regular slice for
$G$ such that for each cluster $Y$ of $\cJ$, we have $\reldeg(Y; R(G)) = \reldeg(Y; G) \pm \eps_k.$
\end{lemma}

Having obtained a regular slice $\cJ$ from the Regular Slice Lemma, we will work within $k$-tuples of clusters of $\cJ$ with respect to which $G$ is both regular and dense. The following definition is useful for keeping track of such $k$-tuples.

\begin{defn}[The $d$-reduced $k$-graph] \label{def:redgraph}
Let $G$ be a $k$-graph and let $\cJ$ be a
$(t_0,t_1,\eps,\eps_k,r)$-regular slice for $G$. Then for $d>0$ we
define the \emph{$d$-reduced $k$-graph of $G$}, denoted $R_d(G)$, to be the $k$-graph whose vertices are the clusters of $\cJ$ and whose edges are all $k$-sets $X$ of clusters of
$\cJ$ such that $G$ is $(\eps_k, r)$-regular with respect to $X$ and $\reld(X) \geq d$. As before, $R_d(G)$ does depend on $\cJ$, but this will always be clear from the context.
\end{defn}

The final lemma we need for working with regular slices is the following, which states that most vertex degrees in the $d$-reduced $k$-graph $R_d(G)$ are similar to those of the weighted reduced $k$-graph $R(G)$ (this is a consequence of the fact that few edges lie in $k$-tuples of clusters which are either irregular or sparse). Again, this is a special case of the full statement from~\cite{ABCM}, which shows that the same is true for a wide range of degree conditions, and also for densities of small subgraphs.

\begin{lemma}\cite[Lemma~8]{ABCM} \label{reducedd+d} 
Let $G$ be a $k$-graph and let $\cJ$ be a
$(t_0,t_1,\eps,\eps_k,r)$-regular slice for~$G$ with $t$
clusters. Then for any cluster $Y$ of $\cJ$ we have
\begin{equation*} \label{eq:reduceddegree} 
\reldeg(Y;R_{d}(G)) \ge \reldeg(Y;R(G)) - d - \zeta(Y),
\end{equation*}
where $\zeta(Y)$ is defined to be the proportion of $k$-sets of clusters $Z$
satisfying $Y \in Z$ which are not $(\eps_k, r)$-regular with respect to $G$.
\end{lemma}

\subsection{Cycle Embedding Lemma}

Having obtained a regular slice for a $k$-graph $G$, we can find tight cycles in $G$ using the Cycle Embedding Lemma from~\cite{ABCM}. This shows that, given a tightly-connected fractional matching $M$ in the $d$-reduced $k$-graph of $G$, we can find a tight cycle $C$ such that for each cluster $X$ of $\cJ$, the proportion of vertices of $X$ covered by $C$ is close to the combined weight in $M$ of edges including $X$. Loosely speaking, this was proved by `winding around' the clusters of each edge of $M$ to form a long tight path, before using the fact that $M$ is tightly connected to extend the path to the clusters of the next edge of $M$, and so forth.

\begin{lemma}\cite[Lemma~9 (Cycle Embedding Lemma)]{ABCM}\label{lem:emb}
  Let $k,r,n_1,t$ be positive integers, and $\psi,d_2,\ldots,d_k,\eps,\eps_k$ be
  positive constants such that $1/d_i \in \NATS$ for each $2 \le i \le k-1$, and such
  that $1/n_1 \ll 1/t$,
\[\frac{1}{n_1} \ll
  \frac{1}{r},\eps\ll\eps_k,d_2,\ldots,d_{k-1}\quad\text{ and }\quad\eps_k\ll
  \psi,d_k,\frac{1}{k}\,.\] Then the following holds for all integers $n\ge
  n_1$. Let~$G$ be a $k$-graph on~$n$ vertices, and $\cJ$ be a $(\cdot,\cdot,\eps,\eps_k,r)$-regular slice for $G$ with $t$ clusters and density vector $(d_{k-1},\ldots,d_{2})$. Suppose that
  $R_{d_k}(G)$ contains a tightly connected fractional matching with total weight~$\mu$. Then $G$ contains a tight cycle of length~$\ell$ for every
      $ \ell\le(1-\psi)k\mu n/t$ that is divisible by~$k$.
\end{lemma}

\section{Proof of Theorem~\ref{main}}\label{sec:mainproof}

Fix $0 < \eta < 1$, and choose $d_3 = \eta/20$ and $\psi = \eta/2$.
Let $\eps_3 \leq \eta^2/10^4$ be sufficiently small to apply Lemma~\ref{lem:emb}
with $k=3$. Also choose functions $\eps(t)$ and $r(t)$ such that for any
$t \in \NATS$ and $d_2$ with $1/d_2 \in \NATS$ we may apply Lemma~\ref{lem:emb}
with $r(t)$ and $\eps(t)$ in place of $r$ and $\eps$ respectively. Now let $t_0 = \max\{100/\eta, 4/\sqrt{\eps_3}\}$, and apply
Lemma~\ref{lem:simpreg} with inputs $t_0, \eps_3, r(.)$ and $\eps(.)$ to obtain
$t_1$ and $n_2$.
For the rest of the proof we write $r$ and $\eps$ for $r(t_1)$
and $\eps(t_1)$ respectively. Finally, choose $n_1 \geq n_2$ sufficiently large
to apply Lemma~\ref{lem:emb} with $t_1$ in place of $t$
and all other constants as above.

Set $n_0 := \max\{n_1 +t_1!, 4t_1!/\eta\}$, and let $H$ be a $3$-graph on $n \geq n_0$ vertices with
$\delta(H) \geq \left(\tfrac{5}{9} + \eta \right)\binom{n}{2}$. It then suffices
to construct a tight cycle in $H$ covering all but at most $\eta n$ vertices of $H$.
For this we begin by arbitrarily deleting up to $t_1!$ vertices of $H$ so that
the number $n'$ of undeleted vertices is divisible by $t_1!$. Note that $n' \geq n_1$,
and let $G$ be the subgraph of $H$ induced by the undeleted vertices. 

We now apply Lemma~\ref{lem:simpreg} to $G$ to obtain a $2$-complex~$\cJ$ which is a
$(t_0,t_1,\eps,\eps_3,r)$-regular slice for $G$ with the property that for any cluster
$Y$ of $\cJ$ we have $\reldeg(Y; R(G)) \geq \reldeg(Y; G) - \eps_3$ (where $R(G)$
is the weighted reduced $3$-graph of $G$). Let $t$ be the number of clusters of $\cJ$.
Then by definition of a regular slice there are at most $\eps_3 \binom{t}{3}$ triples
of clusters of $\cJ$ with respect to which $G$ is not $(\eps_3, r)$-regular, and so all
but at most $\tfrac{1}{2}\sqrt{\eps_3} t$ clusters are \emph{good}, meaning that they lie in fewer
than $2\sqrt{\eps_3} \binom{t}{2}$ such triples. Let~$\C$ be the set of all good clusters,
so $|\C| \geq (1-\tfrac{1}{2}\sqrt{\eps_3}) t$. By arbitrarily removing at most two clusters from $|\C|$ we may assume additionally that $3$ divides $|\C|$, and since $\tfrac{1}{2}\sqrt{\eps_3} t \geq \tfrac{1}{2}\sqrt{\eps_3} t_0 \geq 2$ we still have $|\C| \geq (1-\sqrt{\eps_3}) t$.

Observe that $\delta(G) \geq \delta(H) - t_1! n \geq (\tfrac{5}{9} + \tfrac{4\eta}{5} )\binom{n}{2}$,
so certainly we have $\reldeg(Y; G) \geq \tfrac{5}{9} + \tfrac{4\eta}{5}$.
Since $\eps_3 \leq \tfrac{\eta}{5}$ it follows that
$\reldeg(Y; R(G)) \geq \tfrac{5}{9} + \tfrac{3\eta}{5}$.
We now consider $R_{d_3}(G)$ (the ${d_3}$-reduced $3$-graph of $G$ with respect to $\cJ$);
by Lemma~\ref{reducedd+d} we find that for any cluster $Y \in \C$ we have
\begin{equation*} 
\reldeg(Y; R_{d_3}(G)) \geq \reldeg(Y; R(G)) - d_3 - \frac{2\sqrt{\eps_3} \binom{t}{2}}{\binom{t-1}{2}} \geq \frac{5}{9} + \frac{2\eta}{5},
\end{equation*}
where the second inequality holds since $\binom{t}{2}/\binom{t-1}{2} \leq 2$
(as $t\ge t_0 >100$). It follows that the subgraph $R'$ of $R_{d_3}(G)$ induced by $\C$ has the property that for any
cluster $Y \in \C$ we have 
\begin{equation*} 
\reldeg(Y; R') \geq \frac{5}{9} + \frac{2\eta}{5} - \frac{\sqrt{\eps_3} t^2}{\binom{t-1}{2}} \geq \frac{5}{9} + \frac{\eta}{5}, 
\end{equation*}
or, equivalently,
$$
\delta(R') \geq \left(\frac{5}{9} + \frac{\eta}{5} \right)\binom{|\C|-1}{2}
\ge \frac{5}{9}\binom{|\C|}{2},
$$
where the second inequality holds since $|\C| \geq t_0/2 = 50/\eta$.
We may therefore apply Lemma~\ref{fracmatch} (with $|\C|$ in place of $n$) to obtain a tightly-connected perfect
fractional matching $M$ in $R'$. So $M$ is a tightly-connected matching in $R_d(G)$
of total weight $|\C|/3 \ge (1-\sqrt{\eps_3}) \tfrac{t}{3}$. By Lemma~\ref{lem:emb}
it follows that $G$ (and therefore $H$) admits a tight cycle of length at least
$(1-\psi) \cdot 3 \cdot (1-\sqrt{\eps_3})\tfrac{t}{3} \cdot \tfrac{n'}{t} - 3 = (1-\psi)(1-\sqrt{\eps_3}) n' - 3 \geq (1-\eta) n$. \endproof


\begin{thebibliography}{10}

\bibitem{ABCM} {\sc P.~Allen, J.~B\"ottcher, O.~Cooley and R.~Mycroft},
  {\em Tight cycles and regular slices in dense hypergraphs}, Journal of Combinatorial Theory, Series A, to appear (also see arXiv:1411.4597).

\bibitem{BMSSS16}
{\sc J.~O. Bastos, G.~O. Mota, M.~Schacht, J.~Schnitzer, and F.~Schulenburg},
  {\em Loose {H}amiltonian cycles forced by large (k-2)-degree -- approximate
  version}, arXiv:1603.04180.

\bibitem{BHS13}
{\sc E.~Bu{\ss}, H.~H{\`a}n, and M.~Schacht}, {\em Minimum vertex degree
  conditions for loose {H}amilton cycles in 3-uniform hypergraphs}, Journal of
  Combinatorial Theory, Series B, 103 (2013), pp.~658--678.

\bibitem{CM}
{\sc A.~Czygrinow and T.~Molla}, {\em Tight codegree condition for the
  existence of loose {H}amilton cycles in 3-graphs}, SIAM Journal on Discrete
  Mathematics, 28 (2014), pp.~67--76. 

\bibitem{dirac} 
{\sc G.A~Dirac}, {\em Some theorems on abstract graphs}, Proceedings of the London Mathematical Society, 2 (1952), pp.~69--81.

\bibitem{EG}
{\sc P.~Erd\H{o}s,  T.~Gallai},  {\em On  maximal  paths  and  circuits  of a graph},  Acta  Math.  Sci.  Hungar.,  10 (1959), pp.~337--356.

\bibitem{GaMy16}
{\sc F.~Garbe and R.~Mycroft}, {\em {The complexity of the Hamilton cycle
  problem in hypergraphs of high minimum codegree}}, in 33rd Symposium on
  Theoretical Aspects of Computer Science (STACS 2016), N.~Ollinger and
  H.~Vollmer, eds., vol.~47 of Leibniz International Proceedings in Informatics
  (LIPIcs), Dagstuhl, Germany, 2016, Schloss Dagstuhl--Leibniz-Zentrum f{\"u}r
  Informatik, pp.~38:1--38:13.
  
\bibitem{GPW}
{\sc R.~Glebov, Y.~Person and W.~Weps}, {\em On extremal hypergraphs for Hamiltonian cycles}, European Journal of Combinatorics, 33 (2012), pp.~544--555.

\bibitem{hh-han10}
{\sc H.~H{\`a}n and M.~Schacht}, {\em Dirac-type results for loose {H}amilton
  cycles in uniform hypergraphs}, Journal of Combinatorial Theory, Series B,
  100 (2010), pp.~332--346.

\bibitem{HZ}
{\sc J.~Han and Y.~Zhao}, {\em Minimum codegree threshold for {H}amilton
  $\ell$-cycles in k-uniform hypergraphs}, Journal of Combinatorial Theory,
  Series A, 132 (2015), pp.~194--223.
  
\bibitem{HZ2}
{\sc J.~Han and Y.~Zhao}, {\em Forbidding Hamilton cycles in $k$-uniform hypergraphs}, arXiv:1508.05623.

\bibitem{hh-han14}
{\sc J.~Han and Y.~Zhao}, {\em Minimum vertex degree threshold for loose
  {H}amilton cycles in 3-uniform hypergraphs}, Journal of Combinatorial Theory,
  Series B, 114 (2015), pp.~70--96.

\bibitem{hh-katona99}
{\sc G.~Y. Katona and H.~A. Kierstead}, {\em Hamiltonian chains in
  hypergraphs}, Journal of Graph Theory, 30 (1999), pp.~205--212.

\bibitem{hh-keevash11}
{\sc P.~Keevash, D.~K{\"u}hn, R.~Mycroft, and D.~Osthus}, {\em Loose {H}amilton
  cycles in hypergraphs}, Discrete Mathematics, 311 (2011), pp.~544--559.

\bibitem{KMO14}
{\sc D.~K{\"u}hn, R.~Mycroft, and D.~Osthus}, {\em Hamilton {$\ell$}-cycles in
  uniform hypergraphs}, Journal of Combinatorial Theory, Series A, 117 (2010),
  pp.~910--927.

\bibitem{pp-kuhn06-cherry}
{\sc D.~K{\"u}hn and D.~Osthus}, {\em Loose {H}amilton cycles in 3-uniform
  hypergraphs of high minimum degree}, Journal of Combinatorial Theory, Series
  B, 96 (2006), pp.~767--821.

\bibitem{KOSurvey}
{\sc D.~K{\"u}hn and D.~Osthus}, {\em Hamilton cycles in graphs and
  hypergraphs: an extremal perspective}, Proceedings of the International
  Congress of Mathematicians 2014, Seoul, Korea, 4 (2014), pp.~381--406.

\bibitem{RRSurvey}
{\sc V.~R{\"o}dl and A.~Ruci{\'n}ski}, {\em Dirac-type questions for
  hypergraphs - a survey (or more problems for {E}ndre to solve)}, in An
  Irregular Mind, I.~B{\'a}r{\'a}ny, J.~Solymosi, and G.~S{\'a}gi, eds.,
  vol.~21 of Bolyai Society Mathematical Studies, Springer Berlin Heidelberg,
  2010, pp.~561--590.

\bibitem{RR}
{\sc V.~R{\"o}dl and A.~Ruci{\'n}ski}, {\em Families of triples with high minimum degree are Hamiltonian}, Discuss. Math. Graph Theory, 34 (2014), pp.~361–-381

\bibitem{RRSSz}
{\sc V.~R{\"o}dl, A.~Ruci{\'n}ski, M.~Schacht and E.~Szemer{\'e}di}, {\em On the Hamiltonicity of triple systems with high minimum degree}, arXiv:1605.00773.

\bibitem{hh-rodl06}
{\sc V.~R{\"o}dl, A.~Ruci{\'n}ski, and E.~Szemer{\'e}di}, {\em A {D}irac-type
  theorem for 3-uniform hypergraphs}, Combinatorics, Probability and Computing,
  15 (2006), pp.~229--251.

\bibitem{hh-rodl08}
{\sc V.~R{\"o}dl, A.~Ruci{\'n}ski, and E.~Szemer{\'e}di}, {\em An approximate
  {D}irac-type theorem for {$k$}-uniform hypergraphs}, Combinatorica, 28
  (2008), pp.~229--260.

\bibitem{RoRuSz09}
{\sc V.~R{\"o}dl, A.~Ruci{\'n}ski, and E.~Szemer{\'e}di}, {\em Dirac-type
  conditions for {H}amiltonian paths and cycles in 3-uniform hypergraphs},
  Advances in Mathematics, 227 (2011), pp.~1225--1299.

\bibitem{RS1}
{\sc V.~R{\"o}dl, M. Schacht}, {\em Regular partitions of hypergraphs: regularity lemmas},
  Combinatorics, Probability and Computing, 16 (2007), pp.~833--855.
  
\bibitem{RS2}
{\sc V.~R{\"o}dl, M. Schacht}, {\em Regular partitions of hypergraphs: counting lemmas},
  Combinatorics, Probability and Computing, 16 (2007), pp.~887--901.
  
\bibitem{ZhaoSurvey}
{\sc Y.~Zhao}, {\em Recent advances on {D}irac-type problems for hypergraphs}, in Recent Trends in Combinatorics, the IMA Volumes in Mathematics and its Applications 159. Springer, New York, 2016.

\end{thebibliography}
\end{document}